\documentclass[12pt,a4paper]{amsart}

\usepackage{amsmath,amssymb}
\usepackage[margin=2.5cm]{geometry}
\usepackage{hyperref}

\newtheorem{theorem}{Theorem}[section]
\newtheorem{lemma}[theorem]{Lemma}
\newtheorem{proposition}[theorem]{Proposition}
\newtheorem{corollary}[theorem]{Corollary}
\newtheorem{example}[theorem]{Example}

\theoremstyle{definition}
\newtheorem{definition}[theorem]{Definition}
\theoremstyle{remark}
\newtheorem*{remark}{Remark}

\title{On Numerical Semigroups with Fixed Quotient}

\author{I. Ojeda}
\address{I. Ojeda,  Departamento de Matem\'{a}ticas, Universidad de Extremadura, 06006 Badajoz, Spain. ORCID: \href{https://orcid.org/0000-0003-3173-5934}{0000-0003-3173-5934}}
\email{ojedamc@unex.es}

\author{J.C. Rosales}
\address{J.C. Rosales, Departamento de \'Algebra, Universidad de Granada, 18071 Granada, Spain. ORCID: \href{https://orcid.org/0000-0003-3353-4335}{0000-0003-3353-4335}}
\email{jrosales@ugr.es}

\keywords{numerical semigroup, fixed quotient, multiple, Ap\'ery set, pseudo-Frobenius number.}

\subjclass[2020]{Primary 20M14; Secondary 11D07, 20M25.}

\thanks{The first author was partially supported by project PID2022-138906NB-C21 (MCIN/AEI/10.13039/501100011033, NextGenerationEU/PRTR) and grant GR24068 (Junta de Extremadura, ERDF).}

\begin{document}

\begin{abstract}
Let $\Delta$ be a numerical semigroup and let $d\ge 2$ be an integer. We study the fiber of the quotient map \(S\mapsto S/d\) over $\Delta$. We describe its elements as semigroups of the form $\langle X\rangle+d\Delta$, for suitable finite sets $X\subseteq\Delta$, and then analyze explicit and computable regions of this fiber. In particular, we introduce a family $\Delta_d(a)$ of multiples with prescribed quotient and compute its generators, classical invariants, Ap\'ery sets, and presentations. We also show that this construction preserves Wilf's inequality and controls the depth. Finally, we introduce the $\mathcal{M}_d(\Delta)$-rank, determine its maximal value in the fiber, relate it to the ordinary embedding dimension, characterize the rank-one elements, and give closed formulas for their Frobenius-type invariants and pseudo-Frobenius numbers.
\end{abstract}

\maketitle


\section{Introduction}

Let $\mathbb{N}=\{0,1,2,\dots\}$ and let $\mathbb{Z}$ be the set of integers. A \emph{numerical semigroup} is a submonoid $S\subseteq(\mathbb{N},+)$ with finite complement $\mathbb{N}\setminus S$. We write $m(S)=\min(S\setminus\{0\})$ for the multiplicity, $\operatorname{F}(S)=\max(\mathbb{Z}\setminus S)$ for the Frobenius number, and $\operatorname{g}(S)=|\mathbb{N}\setminus S|$ for the genus. Every numerical semigroup has a unique minimal system of generators, denoted by $\operatorname{msg}(S)$; its cardinality is the embedding dimension $\operatorname{e}(S)$; see \cite[Theorem~2.7]{book}. The classical Frobenius problem asks for expressions of $\operatorname{F}(S)$ and $\operatorname{g}(S)$ in terms of $\operatorname{msg}(S)$; it is completely solved in embedding dimension two and remains widely open in general \cite{sylvester,alfonsin}.

If $A\subseteq\mathbb{N}$ and $d\ge 1$, we write $A/d=\{x\in\mathbb{N}:dx\in A\}$. For a numerical semigroup $S$, the quotient $S/d$ is again a numerical semigroup \cite{proportionally}. We say that $T$ is a \emph{multiple} of $S$ if $T/d=S$ for some $d\ge 1$. Multiples and quotients appear naturally in several constructions and have been studied from different viewpoints, including arithmetic extensions and related closure processes; see, for instance, \cite{cim,australian,ojeda3}.

In this paper we fix a numerical semigroup \(\Delta\neq\mathbb{N}\) and an integer
\(d\ge2\), and we study the fiber
\[
\mathcal{M}_d(\Delta)=\{S:\ S \text{ is a numerical semigroup and } S/d=\Delta\}.
\]
Equivalently, if \(\pi_d\) denotes the map on the set of numerical semigroups defined by \(\pi_d(S)=S/d\), then \(\mathcal{M}_d(\Delta)=\pi_d^{-1}(\Delta)\). In concrete terms, \(S\in\mathcal{M}_d(\Delta)\) if and only if \(S\cap d\mathbb{N}=d\Delta\). Thus the elements of \(S\) divisible by \(d\) are prescribed, and the remaining residue classes modulo \(d\) are to be determined. Our guiding question is structural and effective at the same time: \emph{how explicitly can one describe the whole family $\mathcal{M}_d(\Delta)$, and how do classical invariants vary along it?} Fixing $\Delta$ imposes strong congruence constraints modulo $d$, and these constraints turn out to be rigid enough to yield an explicit model for the fiber.

Our approach has two complementary strands. First, we describe the whole fiber $\mathcal{M}_d(\Delta)$ in structural terms. We prove that its elements are precisely the numerical semigroups of the form $\langle X\rangle+d\Delta$, where $X\subseteq\Delta$ is finite and satisfies explicit compatibility conditions (Theorem~\ref{th:Md-fiber}). This is done by introducing $\mathcal{M}_d(\Delta)$-monoids and exploiting their stability under finite intersections.

Second, we study tractable subfamilies and a rank filtration inside this fiber. The structural description shows that an arbitrary element of $\mathcal{M}_d(\Delta)$ is obtained from $d\Delta$ by adjoining finitely many elements outside $d\Delta$; for arbitrary such choices one should not expect simple closed formulas for the usual invariants. We therefore begin with a particularly uniform construction: given an integer $d\ge 2$ and an element $a\in \Delta \setminus\{0\}$ with $a+1\in \Delta$, we construct a multiple $\Delta_d(a) \in \mathcal{M}_d(\Delta)$ (Theorem~\ref{th:3.1}). This construction adds one generator in each nonzero residue class modulo $d$, and it is explicit enough to yield closed formulas for the embedding dimension, multiplicity, Frobenius number, genus, and related quantities (Theorem~\ref{th:3.3}), as well as descriptions of Ap\'ery sets and presentations. As an application of these formulas, we show that Wilf's inequality is preserved under the family $\Delta_d(a)$ (Corollary~\ref{cor:wilf-preserved}), providing a systematic way to generate further examples from known cases.

In connection with the recent terminology of \cite{australian}, recall that a numerical semigroup is a $k$-quotient if it can be written
as $T/d$ with $\operatorname{e}(T)=k$, and the minimum such $k$ is the \emph{quotient rank}. Our results around Corollary~\ref{cor:embdim-both} show that the set of embedding dimensions occurring among multiples of a fixed semigroup is upward closed above the quotient rank.
Moreover, $k$-quotients admit an equivalent polyhedral description in terms of projections of cones with $k$ extreme rays \cite{cones}.

The same viewpoint also suggests an intrinsic way to measure how many generators are needed outside $d\Delta$ within a fixed fiber. For \(M/d=\Delta\), the \(\mathcal{M}_d(\Delta)\)-rank of \(M\) is defined as the number of generators needed outside \(d\Delta\) in the unique minimal \(\mathcal{M}_d(\Delta)\)-system of generators of \(M\). We show that this rank filtration has a finite top layer: its maximal value on \(\mathcal{M}_d(\Delta)\) is \((d-1)m(\Delta)\). We also relate the relative rank with the ordinary embedding dimension by describing exactly which generators of \(d\Delta\) remain minimal after adjoining the relative generators. Thus the construction \(\Delta_d(a)\) gives a uniform family of \(\mathcal{M}_d(\Delta)\)-rank \(d-1\); in general this is not maximal, but it provides one generator in each nonzero residue class modulo \(d\). In particular, for \(d=2\) it lies in the rank-one layer. The last section analyzes this first genuinely numerical layer of the fiber in full generality. We characterize the rank-one elements (Theorem~\ref{th:rankone-char}) and, in the coprime case $\gcd(x,d)=1$, derive closed formulas for Frobenius-type invariants and pseudo-Frobenius numbers for semigroups of the form $\langle x\rangle+d\Delta$ (Theorems~\ref{th:Fg-rankone} and~\ref{th:PF-rankone}).

The paper is organized as follows. Section~\ref{sect2} establishes the structural description of $\mathcal{M}_d(\Delta)$ via $\mathcal{M}_d(\Delta)$-monoids and finite-intersection properties. Section~\ref{Sect3} develops the explicit construction $\Delta_d(a)$ and computes its invariants, Ap\'ery sets, and presentations. Section~\ref{Sect4} introduces $\mathcal{M}_d(\Delta)$-rank, analyzes rank-one elements, and derives Frobenius-type and pseudo-Frobenius data in that setting.

\section{\texorpdfstring{Structural description of $\mathcal{M}_d(\Delta)$}{Structural description of Md(Delta)}}\label{sect2}

Throughout the rest of the paper, unless explicitly stated otherwise,  $\Delta\neq\mathbb N$ and $d\ge 2$ are fixed as in the introduction.

This section contains the structural core of the paper. We show that the family $\mathcal{M}_d(\Delta)$ admits a concrete description in terms of associated monoids and, ultimately, semigroups of the form $\langle X \rangle+d\Delta$. This characterization is the key step in turning the fixed-quotient problem into an effective one.

\begin{definition}\label{def:2.1}
A submonoid $M$ of $(\mathbb{N},+)$ is called a \emph{$\mathcal{M}_d(\Delta)$-monoid} if $M/d=\Delta$.
\end{definition}

We begin with a basic observation.

\begin{lemma}
If $M$ is a submonoid of $(\mathbb{N},+)$ and $d\ge 1$, then $M/d$ is also a submonoid of $(\mathbb{N},+)$.
\end{lemma}

\begin{proof}
Since $0\in M$, we have $d\cdot 0=0\in M$, and therefore $0\in M/d$. Now let $x,y\in M/d$. Then
$dx\in M$ and $dy\in M$. As $M$ is closed under addition, $d(x+y)=dx+dy\in M$.
Hence $x+y\in M/d$. Therefore $M/d$ is a submonoid of $(\mathbb{N},+)$.
\end{proof}

\begin{proposition}
A submonoid $M$ of $(\mathbb{N},+)$ is a $\mathcal{M}_d(\Delta)$-monoid if and only if $d\Delta\subseteq M$ and
$M\cap d(\mathbb{N}\setminus \Delta)=\varnothing$.
\end{proposition}

\begin{proof}
If $M/d=\Delta$, then $x\in \Delta$ implies $dx\in M$, so $d\Delta\subseteq M$. If $dn\in M$ for some
$n\notin \Delta$, then $n\in M/d=\Delta$, a contradiction. Conversely, assume $d\Delta\subseteq M$ and
$M\cap d(\mathbb{N}\setminus \Delta)=\varnothing$. For every $x\in \Delta$ we have $dx\in M$, hence $x\in M/d$.
If $x\in M/d$, then $dx\in M$, and by the second condition $x\notin \mathbb{N}\setminus \Delta$. Therefore $x\in \Delta$.
\end{proof}

\begin{lemma}
If $A$ and $B$ are submonoids of $(\mathbb{N},+)$, then $A+B=\{a+b : a\in A,\ b\in B\}$ is also a submonoid of $(\mathbb{N},+)$.
\end{lemma}

\begin{proof}
Because $0\in A$ and $0\in B$, we have $0=0+0\in A+B$. If $x,y\in A+B$, then there exist
$a_1,a_2\in A$ and $b_1,b_2\in B$ such that $x=a_1+b_1$ and $y=a_2+b_2$. Since $A$ and $B$ are submonoids,
$a_1+a_2\in A$ and $b_1+b_2\in B$. Therefore $x+y=(a_1+a_2)+(b_1+b_2)\in A+B$.
Thus $A+B$ is a submonoid of $(\mathbb{N},+)$.
\end{proof}

\begin{proposition}\label{prop:2.5}
The family of all $\mathcal{M}_d(\Delta)$-monoids equals
\[
\{A+d\Delta : A \text{ is a submonoid of } (\mathbb{N},+),\ A\cap d(\mathbb{N}\setminus \Delta)=\varnothing\}.
\]
\end{proposition}

\begin{proof}
If $M$ is a $\mathcal{M}_d(\Delta)$-monoid, then $M=M+d\Delta$ because $d\Delta\subseteq M$. Setting
$A=M$, the previous proposition gives the desired inclusion. Conversely, if $A$ is a submonoid with
$A\cap d(\mathbb{N}\setminus \Delta)=\varnothing$, then $A+d\Delta$ is a submonoid by the previous lemma,
contains $d\Delta$. Moreover it still avoids $d(\mathbb{N}\setminus \Delta)$: indeed, if $dn\in A+d\Delta$ with $n\notin\Delta$, then
$dn=a+d\delta$ with $a\in A$ and $\delta\in\Delta$, so $a=d(n-\delta)$. Since $a\in A\subseteq\mathbb{N}$, we must have $n-\delta\in\mathbb{N}$. Since $\delta\in\Delta$ and $\Delta$ is a submonoid, $n\notin\Delta$ forces $n-\delta\notin\Delta$, hence $a\in A\cap d(\mathbb{N}\setminus\Delta)$, a contradiction. Therefore
$(A+d\Delta)\cap d(\mathbb{N}\setminus\Delta)=\varnothing$, and $A+d\Delta$ is an $\mathcal{M}_d(\Delta)$-monoid.
\end{proof}

\begin{proposition}\label{prop:2.6}
The family of $\mathcal{M}_d(\Delta)$-monoids is closed under finite intersections.
\end{proposition}

\begin{proof}
Finite intersections of submonoids of $(\mathbb{N},+)$ are submonoids. Let $M$ and $N$ be $\mathcal{M}_d(\Delta)$-monoids. For $x\in\mathbb{N}$,
\begin{align*}
 x\in (M\cap N)/d & \iff dx\in M\cap N  \iff (dx\in M \text{ and } dx\in N) \\ & \iff (x\in M/d \text{ and } x\in N/d),  
\end{align*}
so $(M\cap N)/d=(M/d)\cap(N/d)=\Delta$. Thus $M\cap N$ is an $\mathcal{M}_d(\Delta)$-monoid. The general finite case follows by induction.
\end{proof}

\begin{lemma}
If $M$ is a $\mathcal{M}_d(\Delta)$-monoid, then $M\subseteq \Delta$.
\end{lemma}

\begin{proof}
If $x\in M$, then $dx\in M$ because $M$ is a submonoid. Hence $x\in M/d=\Delta$.
\end{proof}

\begin{theorem}\label{th:2.8}
The set of all $\mathcal{M}_d(\Delta)$-monoids is
\[
\{\langle X \rangle + d\Delta : X\subseteq \Delta,\ X \text{ finite},\ X\cap d\Delta=\varnothing,
\ \langle X \rangle\cap d(\mathbb{N}\setminus \Delta)=\varnothing\}.
\]
\end{theorem}

\begin{proof}
Let $M$ be a $\mathcal{M}_d(\Delta)$-monoid and define
$X=\{x\in \operatorname{msg}(M): x\notin d\Delta\}$. Then $M=\langle X \rangle+d\Delta$ and $X\subseteq \Delta$ by the previous lemma.
Minimality of $X$ shows $X\cap d\Delta=\varnothing$, while the quotient characterization yields
$\langle X \rangle\cap d(\mathbb{N}\setminus \Delta)=\varnothing$. The converse follows directly from the previous proposition.
\end{proof}

\begin{corollary}\label{cor:2.9}
If $X\subseteq \mathbb{N}$ satisfies $\langle X \rangle\cap d(\mathbb{N}\setminus \Delta)=\varnothing$, then $\langle X \rangle+d\Delta$ is the smallest $\mathcal{M}_d(\Delta)$-monoid containing $X$.
\end{corollary}

\begin{proof}
Since $\langle X \rangle$ is a submonoid of $(\mathbb{N},+)$ and
$\langle X \rangle\cap d(\mathbb{N}\setminus \Delta)=\varnothing$, Proposition~\ref{prop:2.5}
implies that $\langle X \rangle+d\Delta$ is an $\mathcal{M}_d(\Delta)$-monoid. If $M$ is any
$\mathcal{M}_d(\Delta)$-monoid containing $X$, then $\langle X \rangle\subseteq M$ and also
$d\Delta\subseteq M$, so $\langle X \rangle+d\Delta\subseteq M$. Hence $\langle X \rangle+d\Delta$ is
the smallest $\mathcal{M}_d(\Delta)$-monoid containing $X$.
\end{proof}

\begin{lemma}\label{lem:2.10}
Let $X\subseteq \mathbb{N}$. Then $\langle X \rangle+d\Delta$ is a numerical semigroup if and only if
$\operatorname{gcd}(X\cup d\Delta)=1$.
\end{lemma}

\begin{proof}
Since $d\Delta=\langle d\Delta \rangle$, we have $\langle X \rangle+d\Delta=\langle X\cup d\Delta \rangle$. The claim is therefore an immediate consequence of Lemma~2.1 of~\cite{book}.
\end{proof}

Combining the previous theorem with this criterion we obtain the following characterization.

\begin{theorem}\label{th:Md-fiber}
The family $\mathcal{M}_d(\Delta)$ is exactly the set of numerical semigroups of the form $\langle X \rangle+d\Delta$, where $X\subseteq \Delta$ is finite, $X\cap d\Delta=\varnothing$,
$\langle X \rangle\cap d(\mathbb{N}\setminus \Delta)=\varnothing$, and $\operatorname{gcd}(X\cup d\Delta)=1$.
\end{theorem}

\begin{proof}
Let $S\in \mathcal{M}_d(\Delta)$. Then $S$ is, in particular, a $\mathcal{M}_d(\Delta)$-monoid. By Theorem~\ref{th:2.8}, there exists a finite set $X\subseteq \Delta$ such that $S=\langle X \rangle+d\Delta$,
with $X\cap d\Delta=\varnothing$ and $\langle X \rangle\cap d(\mathbb{N}\setminus \Delta)=\varnothing$. Since $S$ is a numerical semigroup, Lemma~\ref{lem:2.10} yields $\operatorname{gcd}(X\cup d\Delta)=1$.

Conversely, let $X\subseteq \Delta$ be finite and assume that
$X\cap d\Delta=\varnothing$, $\langle X \rangle\cap d(\mathbb{N}\setminus \Delta)=\varnothing$, and
$\operatorname{gcd}(X\cup d\Delta)=1$. By Theorem~\ref{th:2.8}, the set $\langle X \rangle+d\Delta$ is a $\mathcal{M}_d(\Delta)$-monoid, that is, its quotient by $d$ is $\Delta$. By Lemma~\ref{lem:2.10}, the gcd condition implies that $\langle X \rangle+d\Delta$ is a numerical semigroup. Hence $\langle X \rangle+d\Delta\in \mathcal{M}_d(\Delta)$.
\end{proof}

Theorem~\ref{th:Md-fiber} gives a structural description of the entire fiber
$\mathcal{M}_d(\Delta)$. In general, however, the freedom in the choice of $X$ makes it
difficult to obtain closed formulas for invariants directly from this description. We next
focus on a particularly uniform subfamily, obtained by adjoining one element in each
nonzero residue class modulo $d$. This construction is still flexible enough to produce
many multiples, but rigid enough to allow explicit computations of generators, numerical
invariants, Ap\'ery sets, and presentations.

\section{Explicit constructions and invariants}\label{Sect3}

The structural description of Section~\ref{sect2} reduces the fixed-quotient problem to
understanding how finite sets of generators may be adjoined to $d\Delta$. In this section
we study a particularly uniform way of doing this: we construct a multiple $\Delta_d(a)$ by adjoining, after a fixed threshold, one generator
in each nonzero residue class modulo $d$. This family is explicit enough for the standard
invariants to be computed directly, and it will later serve as a motivating example for the
rank filtration introduced in Section~\ref{Sect4}.

\subsection{\texorpdfstring{The construction of $\Delta_d(a)$ and its invariants}{The construction of Deltad(a) and its invariants}}

Let us begin with a particularly useful family. Throughout this section, let
\(a\in\Delta\setminus\{0\}\) satisfy \(a+1\in\Delta\). We define
\[
\Delta_d(a)= d\Delta \,\cup\, \bigl(\{da+1,da+2,\dots,da+d-1\}+d\Delta\bigr).
\]
Equivalently, $\Delta_d(a)= d\Delta \cup \bigcup_{i=1}^{d-1} (da+i+d\Delta)$.

\begin{theorem}\label{th:3.1}
Let $a\in \Delta\setminus\{0\}$ with $a+1\in \Delta$. Then $\Delta_d(a)$ is a numerical semigroup and
\[
\frac{\Delta_d(a)}{d}=\Delta,\ \text{that is}\ \Delta_d(a) \in \mathcal{M}_d(\Delta).
\]
Moreover, if $\operatorname{msg}(\Delta)=\{n_1,\dots,n_e\}$, then
\begin{equation}\label{ecu:msg}
\Delta_d(a)=\bigl\langle dn_1,\dots,dn_e,\ da+1,\dots,da+d-1\bigr\rangle.
\end{equation}
In particular, $\Delta_d(a)$ is a multiple of $\Delta$.
\end{theorem}

\begin{proof}
We first prove that $\Delta_d(a)$ is a submonoid of $(\mathbb{N},+)$. Clearly
$0\in d\Delta\subseteq \Delta_d(a)$. Let $x,y\in \Delta_d(a)$. If both lie in $d\Delta$, then
$x+y\in d\Delta$. If one lies in $d\Delta$ and the other in $da+i+d\Delta$, with
$i\in\{1,\dots,d-1\}$, then their sum lies again in $da+i+d\Delta$.

It remains to consider the case $x=da+i+ds$ and $y=da+j+dt$, with
$s,t\in \Delta$ and $i,j\in\{1,\dots,d-1\}$. If $i+j<d$, then
$x+y=da+(i+j)+d(a+s+t)\in da+(i+j)+d\Delta$. If $i+j=d$, then
$x+y=d(a+1+a+s+t)\in d\Delta$. Finally, if $i+j>d$, then
$x+y=da+(i+j-d)+d(a+1+s+t)\in da+(i+j-d)+d\Delta$. Hence $\Delta_d(a)$ is a submonoid.

Now let $\operatorname{msg}(\Delta)=\{n_1,\dots,n_e\}$. Since
$d\Delta=\langle dn_1,\dots,dn_e\rangle$, each set $da+i+d\Delta$ is contained in
$\langle dn_1,\dots,dn_e,da+1,\dots,da+d-1\rangle$. Hence $\Delta_d(a)$ is contained in
the right-hand side of \eqref{ecu:msg}. Conversely, all the displayed generators belong
to $\Delta_d(a)$, and since $\Delta_d(a)$ is a submonoid, the semigroup generated by them is
contained in $\Delta_d(a)$. Thus \eqref{ecu:msg} holds.

Finally, since $\gcd(n_1,\dots,n_e)=1$, we have $\gcd(dn_1,\dots,dn_e)=d$, and
$\gcd(d,da+1)=1$. Therefore the greatest common divisor of the generators in
\eqref{ecu:msg} is $1$, and $\Delta_d(a)$ is a numerical semigroup by
\cite[Lemma~2.1]{book}.

Moreover, every element of $da+i+d\Delta$ is congruent to $i\in\{1,\dots,d-1\}$ modulo
$d$, so the only multiples of $d$ in $\Delta_d(a)$ are those in $d\Delta$. Hence
$x\in \Delta_d(a)/d$ if and only if $dx\in d\Delta$, that is, if and only if $x\in \Delta$. Thus
$\Delta_d(a)/d=\Delta$.
\end{proof}

The first step is to understand the generators of $\Delta_d(a)$; once these are known, the standard invariants follow with little extra work.

\begin{proposition}\label{prop:2}
Let $\operatorname{msg}(\Delta)=\{n_1,\dots,n_e\}$. Then $\{dn_1,\dots,dn_e,da+1,da+2,\dots,da+d-1\}$ is the minimal system of generators of $\Delta_d(a)$. In particular, $\operatorname{e}(\Delta_d(a))=\operatorname{e}(\Delta)+d-1$,
where $\operatorname{e}(\Delta)$ denotes the embedding dimension of $\Delta$.
\end{proposition}

\begin{proof}
By \eqref{ecu:msg}, $\Delta_d(a)=\langle dn_1,\dots,dn_e,\,da+1,\dots,da+d-1\rangle$.
We now prove that this generating set is minimal.

First fix $i\in\{1,\dots,e\}$ and suppose, for contradiction, that
\[
dn_i=\sum_{j\neq i}\lambda_j\,dn_j+\sum_{r=1}^{d-1}\mu_r(da+r),
\qquad \lambda_j,\mu_r\in\mathbb{N}.
\]
If $\sum_{r=1}^{d-1}\mu_r=0$, then
\[
n_i=\sum_{j\neq i}\lambda_j n_j,
\]
contradicting $n_i\in\operatorname{msg}(\Delta)$. Hence we may assume that
$\sum_{r=1}^{d-1}\mu_r>0$.

Set $k=\sum_{r=1}^{d-1}\mu_r$ and $R=\sum_{r=1}^{d-1}\mu_r r$.
Reducing modulo $d$ yields $R\equiv 0\pmod d$, so $R=qd$ for some $q\ge 1$.
Note that necessarily $k\ge 2$, since otherwise $R=r$ with $1\le r\le d-1$ cannot
be divisible by $d$. Since $k\le R\le k(d-1)$, we get $1\le q\le k-1$.

Dividing the original equality by $d$ gives
\[
n_i=\sum_{j\neq i}\lambda_j n_j+ka+q.
\]
Moreover, $ka+q=(k-q)a+q(a+1)$, and since $a,a+1\in \Delta$ with
$1\le q\le k-1$, this expresses $ka+q$ as a sum of nonzero elements of $\Delta$.
Hence $n_i$ is a sum of nonzero elements of $\Delta$, contradicting
$n_i\in\operatorname{msg}(\Delta)$. Therefore each $dn_i$ is indispensable.

Next fix $i\in\{1,\dots,d-1\}$. Any element of $\langle dn_1,\dots,dn_e\rangle$
is $0$ modulo $d$, so any expression of $da+i$ must involve some generator $da+j$.
If it involves at least two such generators, counting multiplicity, then its value is at
least $2(da+1)>da+i$, impossible. Thus
\[
da+i=(da+j)+\sum_{r=1}^e\lambda_r dn_r
\]
for some $j\in\{1,\dots,d-1\}$ and $\lambda_r\in\mathbb{N}$. Reducing modulo $d$
gives $i\equiv j\pmod d$, hence $j=i$, and then
$\sum_{r=1}^e\lambda_rdn_r=0$, so all $\lambda_r=0$. Therefore $da+i$ is also
indispensable.

Consequently, $\{dn_1,\dots,dn_e,da+1,\dots,da+d-1\}$ is the minimal system of
generators of $\Delta_d(a)$, and $\operatorname{e}(\Delta_d(a))=\operatorname{e}(\Delta)+d-1$.
\end{proof}

This raises the natural question of which embedding dimensions occur among multiples of $\Delta$.
The next corollary shows that realizability is upward closed once one multiple is known.

\begin{corollary}\label{cor:embdim-both}
Let $\Delta$ be a numerical semigroup.
\begin{enumerate}
\item For every integer $k\ge \operatorname{e}(\Delta)$ there exist an integer $d\ge 1$ and a numerical
semigroup $S$ such that $S/d=\Delta$ and $\operatorname{e}(S)=k$. More precisely, one may take
$d=k-(\operatorname{e}(\Delta)-1)$.
\item Let
$
e_{\min}=\min\{\operatorname{e}(T): T \text{ is a multiple of } \Delta\}.
$
Then every integer $k\ge e_{\min}$ occurs as the embedding dimension of some multiple of
$\Delta$.
\end{enumerate}
\end{corollary}

\begin{proof}
(1) If $k=\operatorname{e}(\Delta)$, take $S=\Delta$ and $d=1$. Otherwise $k>\operatorname{e}(\Delta)$, set
$d=k-(\operatorname{e}(\Delta)-1)\ge 2$, choose $a\in\Delta\setminus\{0\}$ with $a+1\in\Delta$, and take
$S=\Delta_d(a)$. Then $S/d=\Delta$ and $\operatorname{e}(S)=\operatorname{e}(\Delta)+d-1=k$.

(2) By definition there exists a multiple $T$ of $\Delta$ with $\operatorname{e}(T)=e_{\min}$.
Applying (1) to $T$ yields, for each $k\ge e_{\min}$, a multiple $S$ of $T$ with $\operatorname{e}(S)=k$.
Since quotients compose, a multiple of a multiple is again a multiple, hence $S$ is a
multiple of $\Delta$.
\end{proof}

\begin{remark}\label{rem:quotient-rank-cones}
Following \cite{australian}, a numerical semigroup $\Delta$ is called a \emph{$k$-quotient}
if $\Delta=T/d$ for some numerical semigroup $T$ with $\operatorname{e}(T)=k$ and some $d\ge 1$; the
minimum such $k$ is the \emph{quotient rank} of $\Delta$.
In this terminology, $e_{\min}$ equals the quotient rank of $\Delta$, and
Corollary~\ref{cor:embdim-both} says that $k$-quotientability is upward closed in $k$:
if $\Delta$ is a $k_0$-quotient, then it is a $k$-quotient for every $k\ge k_0$.

Moreover, by \cite{cones}, $k$-quotients are exactly the $k$-ray-normalescent numerical
semigroups, i.e.\ those arising as a projection of the integer points in a rational
polyhedral cone with $k$ extreme rays. Consequently, Corollary~\ref{cor:embdim-both} also implies
that $k$-ray-normalescence is upward closed in $k$.
In contrast, computing the threshold $e_{\min}$ (i.e., the quotient rank) remains open in
general; see, for instance, \cite{ojeda3}.
\end{remark}

\subsection{Frobenius-type invariants and Wilf's conjecture}
We now compute closed formulas for the Frobenius-type and related invariants of $\Delta_d(a)$ in terms of $\Delta$, $d$, and $a$.

\begin{theorem}\label{th:3.3}
For the numerical semigroup $\Delta_d(a)$ one has
\begin{align*}
m(\Delta_d(a))&=d\,m(\Delta),\\
\operatorname{F}(\Delta_d(a))&=d\operatorname{F}(\Delta)+da+d-1,\\
\operatorname{g}(\Delta_d(a))&=d\,\operatorname{g}(\Delta)+(d-1)a,\\
\operatorname{n}(\Delta_d(a))&=d\,\operatorname{n}(\Delta)+a,\\
\operatorname{e}(\Delta_d(a))&=\operatorname{e}(\Delta)+d-1,
\end{align*}
where $\operatorname{n}(\Delta)=|\{s\in \Delta:s<\operatorname{F}(\Delta)\}|$.
\end{theorem}

\begin{proof}
By Proposition \ref{prop:2}, the formulas for multiplicity and embedding dimension are immediate from the construction and the minimal generators above.

The Frobenius number is obtained by observing that, by the definition of \(\Delta_d(a)\), the elements congruent to \(0\) modulo \(d\) are exactly those of \(d\Delta\), while the elements congruent to \(i\in\{1,\dots,d-1\}\) are exactly those of \(da+i+d\Delta\). Hence the largest gap in residue class \(i\) is \(da+i+d\operatorname{F}(\Delta)\), and so \[ \operatorname{F}(\Delta_d(a))=d\operatorname{F}(\Delta)+da+d-1. \]

To compute the genus, we count gaps by congruence classes using the same residue-class description. Write $\operatorname{gaps}(\Delta)=\mathbb{N}\setminus \Delta$. In residue class $0$, integers have the form $dk$, and since the only multiples of $d$ in $\Delta_d(a)$ are those in $d\Delta$, we have $dk\in \Delta_d(a)$ if and only if $k\in \Delta$, so this class contributes exactly $\operatorname{g}(\Delta)$ gaps. For $i\in\{1,\dots,d-1\}$ write $x=dk+i$. Then $x\in \Delta_d(a)$ if and only if $x\in da+i+d\Delta$, equivalently $k-a\in \Delta$ (and necessarily $k\ge a$). Therefore the gaps in this class are precisely the $a$ integers with $k<a$, together with those with $k\ge a$ and $k-a\in\operatorname{gaps}(\Delta)$, which contribute $\operatorname{g}(\Delta)$ further gaps. Thus each nonzero residue class contributes $a+\operatorname{g}(\Delta)$ gaps, and summing gives
\[
\operatorname{g}(\Delta_d(a))=\operatorname{g}(\Delta)+(d-1)\bigl(a+\operatorname{g}(\Delta)\bigr)=d\,\operatorname{g}(\Delta)+(d-1)a.
\]

Finally, $\operatorname{n}(\Delta_d(a))=\operatorname{F}(\Delta_d(a))+1-\operatorname{g}(\Delta_d(a))$, and substituting the expressions of $\operatorname{F}(\Delta_d(a))$ and $\operatorname{g}(\Delta_d(a))$ yields $\operatorname{n}(\Delta_d(a))=d\,\operatorname{n}(\Delta)+a$.
\end{proof}

\begin{example}
If $\Delta=\langle 4,5\rangle$, $d=3$, and $a=4$, then $\Delta_3(4)=\langle 12,13,14,15\rangle$. The previous theorem gives $\operatorname{F}(\Delta_3(4))=3\operatorname{F}(\langle 4,5\rangle)+3\cdot 4+2=47$ and $\operatorname{g}(\Delta_3(4))=3\operatorname{g}(\langle 4,5\rangle)+2\cdot 4=26$.
\end{example}

One immediate consequence concerns Wilf's conjecture, which we now recall.

\begin{remark}[Wilf's conjecture]
Let $S$ be a numerical semigroup with embedding dimension $\operatorname{e}(S)$ and conductor $\operatorname{c}(S)=\operatorname{F}(S)+1$. 
Wilf's conjecture (1978) asserts that
\[
\operatorname{e}(S)\,\operatorname{n}(S)\ge \operatorname{c}(S)
\]
for every numerical semigroup $S$. This remains open in general; see, for instance, \cite{EliahouWilf}.
\end{remark}

\begin{corollary}\label{cor:wilf-preserved}
If $\Delta$ satisfies Wilf's conjecture, then $\Delta_d(a)$ also satisfies Wilf's conjecture.
\end{corollary}

\begin{proof}
Set $S = \Delta_d(a)$. Using $\operatorname{c}(\Delta)=\operatorname{F}(\Delta)+1$ and the formulas from the previous theorem, $\operatorname{e}(S)=\operatorname{e}(\Delta)+d-1$, $\operatorname{n}(S)=d\,\operatorname{n}(\Delta)+a$, $\operatorname{c}(S)=\operatorname{F}(S)+1=d\operatorname{F}(\Delta)+da+d$, we compute
\[
\operatorname{e}(S)\,\operatorname{n}(S)-\operatorname{c}(S)
= d\bigl(\operatorname{e}(\Delta)\operatorname{n}(\Delta)-\operatorname{c}(\Delta)\bigr)+d(d-1)\operatorname{n}(\Delta)+a\bigl(\operatorname{e}(\Delta)-1\bigr).
\]
If $\Delta$ satisfies Wilf's conjecture, then $\operatorname{e}(\Delta)\operatorname{n}(\Delta)-\operatorname{c}(\Delta)\ge 0$, and clearly $d(d-1)\operatorname{n}(\Delta)\ge 0$ and $a(\operatorname{e}(\Delta)-1)\ge 0$. Hence
$\operatorname{e}(S)\,\operatorname{n}(S)\ge \operatorname{c}(S)$, so $S$ satisfies Wilf's conjecture.
\end{proof}

\begin{remark}
Recall that the depth of a numerical semigroup $S$, in the sense of Eliahou \cite{EliahouFromentin}, is
\[
q(S)=\left\lceil \frac{\operatorname{c}(S)}{m(S)}\right\rceil,
\]
where $\operatorname{c}(S)=\operatorname{F}(S)+1$ is the conductor. This invariant is relevant in connection with
Wilf's conjecture: for instance, Wilf's conjecture is known for numerical semigroups of
depth at most three (\cite{EliahouWilf}).

For the family constructed above, Theorem~\ref{th:3.3} gives
\[
\operatorname{c}(\Delta_d(a))=d(\operatorname{c}(\Delta)+a)
\quad\text{and}\quad
m(\Delta_d(a))=d\,m(\Delta).
\]
Hence
\[
q(\Delta_d(a))
=
\left\lceil \frac{\operatorname{c}(\Delta)+a}{m(\Delta)}\right\rceil.
\]
In particular, the depth of $\Delta_d(a)$ is independent of $d$ and satisfies
$q(\Delta_d(a))\ge q(\Delta)$. Thus the construction may produce Wilf semigroups of larger
depth from semigroups already known to satisfy Wilf's conjecture.
\end{remark}

\subsection{Ap\'ery sets}

The formulas above are obtained by counting residue classes modulo $d$. Ap\'ery sets
provide a more refined way to package the same information. We now describe the Ap\'ery
set of $\Delta_d(a)$ with respect to its multiplicity in terms of the Ap\'ery set of $\Delta$.

Recall that, for a numerical semigroup $S$ and a nonzero element $m\in S$, the \emph{Ap\'ery set} of $S$ with respect to $m$ is
\[
\operatorname{Ap}(S,m)=\{\,s\in S:\ s-m\notin S\,\}.
\]
By \cite[Lemma~2.4]{book}, it contains exactly one element in each congruence class modulo $m$; in particular, the cardinality of $\operatorname{Ap}(S,m)$ is $m$.

A useful reduction lemma for Ap\'ery sets is the following (see, e.g., \cite[Theorem 27]{australian}): if $d$ divides $m$ and $m\in S\setminus\{0\}$, then \[\operatorname{Ap}\left(\frac{S}{d},\frac{m}{d}\right)=\left\{\frac{w}{d}: w\in \operatorname{Ap}(S,m),\ w\equiv 0 \pmod d\right\}.\] Applying this identity to $\Delta_d(a)$ identifies the elements of the Ap\'ery set lying in the congruence class $0$ modulo $d$. The remaining congruence classes can be described directly from the definition of $\Delta_d(a)$.

\begin{proposition}
Let $m=m(\Delta)$ and set $S=\Delta_d(a)$. Then
\[
\operatorname{Ap}(S,dm)
=
d\,\operatorname{Ap}(\Delta,m)
\ \cup\
\bigcup_{i=1}^{d-1}
\bigl(da+i+d\,\operatorname{Ap}(\Delta,m)\bigr),
\]
and this union is disjoint.
\end{proposition}

\begin{proof}
First note that the elements of $S$ congruent to $0$ modulo $d$ are exactly those of $d\Delta$, and, for
$i\in\{1,\dots,d-1\}$, the elements of $S$ congruent to $i$ modulo $d$ are exactly those of $da+i+d\Delta$.

Let $w\in\operatorname{Ap}(\Delta,m)$. Then $dw\in d\Delta\subseteq S$ and
$dw-dm=d(w-m)\notin S$, because $w-m\notin \Delta$ and the elements of $S$ congruent to $0$ modulo $d$ are exactly those of $d\Delta$. Hence $d\,\operatorname{Ap}(\Delta,m)\subseteq\operatorname{Ap}(S,dm)$.

Now fix $i\in\{1,\dots,d-1\}$ and $w\in\operatorname{Ap}(\Delta,m)$. Then $da+i+dw\in da+i+d\Delta\subseteq S$, and $
(da+i+dw)-dm=da+i+d(w-m)\notin S$, 
because the elements of $S$ congruent to $i$ modulo $d$ are exactly those of $da+i+d\Delta$, and
$w-m\notin \Delta$. Thus $
da+i+d\,\operatorname{Ap}(\Delta,m)\subseteq\operatorname{Ap}(S,dm)
$
for every $i\in\{1,\dots,d-1\}$.

Conversely, let $x\in\operatorname{Ap}(S,dm)$, and write $x\equiv i\pmod d$ with
$i\in\{0,1,\dots,d-1\}$.

If $i=0$, then $x=ds$ for some $s\in \Delta$. Since
$
x-dm=d(s-m)\notin S$,
we get $s-m\notin \Delta$, and hence $s\in\operatorname{Ap}(\Delta,m)$. Therefore
$
x\in d\,\operatorname{Ap}(\Delta,m)$.

If $i\in\{1,\dots,d-1\}$, then $x=da+i+ds$ for some $s\in \Delta$. Since $x-dm=da+i+d(s-m)\notin S$, we get $s-m\notin \Delta$, and hence $s\in\operatorname{Ap}(\Delta,m)$. Therefore
$x\in da+i+d\,\operatorname{Ap}(\Delta,m)$.

Finally, the sets corresponding to different residues modulo $d$ are disjoint. Within a fixed residue class, equality
$da+i+dw=da+i+dw'$ implies $w=w'$, and similarly in the class $0$. Hence the union is disjoint.
\end{proof}

\subsection{Presentations}

The construction $\Delta_d(a)$ is explicit not only at the level of numerical invariants. Since
its minimal generators consist of the dilated generators of $\Delta$ together with the new
generators $da+1,\dots,da+d-1$, one can also lift the factorization structure of $\Delta$.
We now make this precise by describing a presentation of $\Delta_d(a)$ in terms of a
presentation of $\Delta$ and the quadratic relations among the new generators.

Let $S$ be a numerical semigroup, let $\operatorname{msg}(S)=\{n_1,\dots,n_e\}$, and let
$\varphi:\mathbb{N}^e\to S$, $(x_1,\dots,x_e)\mapsto x_1n_1+\cdots+x_en_e$.
The kernel congruence of $\varphi$ is finitely generated, and any finite generating system of this congruence is called a presentation of $S$; see \cite[Chapter~7, Section~1]{book}.

For the semigroup $\Delta_d(a)$ the minimal generators are $dn_1,\dots,dn_e,da+1,\dots,da+d-1$.
Besides the relations coming from a presentation of $\Delta$, one has quadratic relations among the new generators induced by
\begin{align*}
(da+i)+(da+j)&=da+(i+j)+da &&\text{if } i+j<d,\\
(da+i)+(da+j)&=d(a+1)+da &&\text{if } i+j=d,\\
(da+i)+(da+j)&=d(a+1)+(da+i+j-d) &&\text{if } i+j>d.
\end{align*}

Fix factorizations $a=\sum_{r=1}^e u_r n_r$ and
$a+1=\sum_{r=1}^e v_r n_r$ in $\Delta$, and set
$u=(u_1,\dots,u_e)$ and $v=(v_1,\dots,v_e)$. We view \(u\) and \(v\) as elements
of \(\mathbb{N}^{e+d-1}\) by adding \(d-1\) zero coordinates.

Let \(f_i\) denote the coordinate vector corresponding to the generator \(da+i\).
For \(1\le i\le j\le d-1\), the lifted quadratic relations are
\[
f_i+f_j \sim
\begin{cases}
u+f_{i+j}, & \text{if } i+j<d,\\
u+v, & \text{if } i+j=d,\\
v+f_{i+j-d}, & \text{if } i+j>d.
\end{cases}
\]
Let \(\rho_1\) denote the finite family of all these relations.

Now let $\sigma\subseteq \mathbb{N}^e\times \mathbb{N}^e$ be a presentation of $\Delta$. Viewing $\mathbb{N}^e$ as the submonoid of $\mathbb{N}^{e+d-1}$ given by the first $e$ coordinates, define the \emph{lift} of $\sigma$ to $\Delta_d(a)$ as
\[
\rho_2=\{((x,0),(y,0))\in \mathbb{N}^{e+d-1}\times \mathbb{N}^{e+d-1} : (x,y)\in \sigma\},
\]
where $0$ denotes the zero vector in $\mathbb{N}^{d-1}$.

\begin{theorem}
Let $\sigma$ be a presentation of $\Delta$. If $\rho_1$ denotes the family of lifted quadratic relations among the generators $da+1,\dots,da+d-1$, and $\rho_2$ is the lift of $\sigma$ to the generators of $\Delta_d(a)$, then $\rho_1\cup \rho_2$ is a presentation of $\Delta_d(a)$.
\end{theorem}

\begin{proof}
Let \(\pi:\mathbb{N}^{e+d-1}\to\Delta_d(a)\) be the factorization homomorphism associated with the generators $dn_1,\dots,dn_e,da+1,\dots,da+d-1$. By the standard characterization of presentations, it is enough to show that the relations in \(\rho_1\cup\rho_2\) connect any two factorizations of the same element.

Let \((\alpha,\beta)\in\ker(\pi)\). For a factorization \(\gamma\), denote by \(\nu(\gamma)\) the number of occurrences, counted with multiplicity, of the new generators \(da+1,\dots,da+d-1\). Each relation in \(\rho_1\) can be used to replace two new generators by an expression involving at most one new generator, and hence decreases \(\nu\). Therefore every factorization is congruent modulo \(\rho_1\) to one with \(\nu\le1\). Thus, after replacing \(\alpha\) and \(\beta\) by congruent factorizations modulo \(\rho_1\), we may assume that \(\nu(\alpha)\le1\) and \(\nu(\beta)\le1\).

If \(\nu(\alpha)=0\), then \(\pi(\alpha)\in d\Delta\). Since
\(\pi(\alpha)=\pi(\beta)\), also \(\nu(\beta)=0\). Dividing the equality \(\pi(\alpha)=\pi(\beta)\) by \(d\), we obtain an equality of factorizations in \(\Delta\). Hence \(\alpha\) and \(\beta\) are congruent modulo the lifted relations \(\rho_2\).

If \(\nu(\alpha)=1\), then also \(\nu(\beta)=1\), and the unique new generator appearing  in both factorizations has the same residue modulo \(d\). Hence both factorizations use the same generator \(da+i\), for some \(i\in\{1,\dots,d-1\}\). Removing this common coordinate from both factorizations, the remaining parts give an equality in \(d\Delta\). Dividing by \(d\), the corresponding factorizations in \(\Delta\) are congruent modulo \(\sigma\), and therefore the original factorizations are congruent modulo \(\rho_2\).

Consequently, after reduction by relations in \(\rho_1\), any two factorizations of the same element are connected by relations in \(\rho_2\). Since the reductions themselves use only relations in \(\rho_1\), the original factorizations are connected by relations
in \(\rho_1\cup\rho_2\). Thus \(\rho_1\cup\rho_2\) is a presentation of \(\Delta_d(a)\).
\end{proof}

We illustrate the previous theorem with a concrete example.

\begin{example}
Let $\Delta=\langle n_1=5,\; n_2=7,\; n_3=9\rangle$, $d=4$, and $a=9$. Then
\[
\Delta_4(9)=\langle 20,28,36,37,38,39\rangle.
\]
A presentation of $\Delta=\langle 5,7,9\rangle$ is given by the following pairs in $\mathbb{N}^3\times\mathbb{N}^3$:
\[
\sigma=\{((5,0,0),(0,1,2)),\ ((0,2,0),(1,0,1)),\ ((0,0,3),(4,1,0))\},
\]
corresponding to the equalities $5\cdot 5=7+2\cdot 9$, $2\cdot 7=5+9$, and $3\cdot 9=4\cdot 5+7$. 
The quadratic relations among $37$, $38$, and $39$ come from rewriting $37+37,\,37+38,\,37+39,\,38+38,\,38+39,$ and $39+39$ via the identities in the text and expressing the multiples $36$ and $40$
in terms of $20$, $28$, and $36$.
\end{example}

\section{\texorpdfstring{Rank in $\mathcal{M}_d(\Delta)$}{Rank in Md(Delta)}}\label{Sect4}

The construction of Section~\ref{Sect3} adjoins $d-1$ generators outside $d\Delta$, one in each nonzero residue class modulo $d$. More generally, by Theorem~\ref{th:Md-fiber}, every element of $\mathcal{M}_d(\Delta)$ is obtained from $d\Delta$ by adjoining finitely many elements outside $d\Delta$.  We use this to introduce the $\mathcal{M}_d(\Delta)$-rank and study the induced filtration, with special emphasis on rank one.

\subsection{\texorpdfstring{$\mathcal{M}_d(\Delta)$-generators and rank}{Md(Delta)-generators and rank}}

Recall from Definition~\ref{def:2.1} that a submonoid $M\subseteq(\mathbb{N},+)$ is an $\mathcal{M}_d(\Delta)$-monoid if $M/d=\Delta$. We will repeatedly use that $\mathcal{M}_d(\Delta)$ is closed under finite intersections and that $(S\cap T)/d=(S/d)\cap(T/d)$ for numerical semigroups $S,T$; see Proposition~\ref{prop:2.6}.

\begin{proposition}\label{prop:arb-inter}
A submonoid $M\subseteq(\mathbb{N},+)$ is a nonempty intersection of elements of $\mathcal{M}_d(\Delta)$ if and only if $M/d=\Delta$.
\end{proposition}

\begin{proof}
If $M=\bigcap_{i\in I}S_i$ with $I\neq\varnothing$ and
$S_i\in\mathcal{M}_d(\Delta)$ for all $i\in I$, then for $x\in\mathbb{N}$ one has
\[
x\in M/d \iff dx\in M \iff dx\in S_i \text{ for all } i\in I
\iff x\in S_i/d \text{ for all } i\in I,
\]
so $M/d=\bigcap_{i\in I}(S_i/d)=\Delta$.

Conversely, assume $M/d=\Delta$. For each integer $n>d\operatorname{F}(\Delta)$ set
$S_n:=M\cup\{n,n+1,\dots\}$.
Then $S_n$ is a numerical semigroup and $S_n/d=\Delta$: clearly
$\Delta\subseteq S_n/d$, and if $x\notin\Delta$, then $x\le \operatorname{F}(\Delta)$, hence
$dx\le d\operatorname{F}(\Delta)<n$, so $dx\notin\{n,n+1,\dots\}$ and $dx\notin M$, whence
$x\notin S_n/d$.
Finally,
\[
M=\bigcap_{n>d\operatorname{F}(\Delta)}S_n,
\]
since each $S_n$ contains $M$ and, for $x\notin M$, taking
$n>\max\{d\operatorname{F}(\Delta),x\}$ yields $x\notin S_n$.
\end{proof}

\begin{definition}
A subset $X\subseteq\mathbb{N}$ is called an $\mathcal{M}_d(\Delta)$-set if it is contained in some element of $\mathcal{M}_d(\Delta)$. In that case, we denote by $\mathcal{M}_d(\Delta)[X]$ the intersection of all $S\in\mathcal{M}_d(\Delta)$ with $X\subseteq S$.
\end{definition}

By Corollary~\ref{cor:2.9} and Proposition~\ref{prop:arb-inter}, a subset
$X\subseteq\mathbb N$ is an $\mathcal{M}_d(\Delta)$-set if and only if
$\langle X\rangle\cap d(\mathbb N\setminus\Delta)=\varnothing$.

\begin{proposition}\label{prop:Md-closure}
If $X$ is an $\mathcal{M}_d(\Delta)$-set, then
$\mathcal{M}_d(\Delta)[X]=\langle X\rangle+d\Delta$.
\end{proposition}

\begin{proof}
Let $N=\mathcal{M}_d(\Delta)[X]$, the intersection of all
$S\in\mathcal{M}_d(\Delta)$ with $X\subseteq S$.
Since each such $S$ contains $\langle X\rangle$ and $d\Delta$, we have
$\langle X\rangle+d\Delta\subseteq N$.

Conversely, since $X$ is an $\mathcal{M}_d(\Delta)$-set, we have
$\langle X\rangle\cap d(\mathbb{N}\setminus\Delta)=\varnothing$.
Hence, by Corollary~\ref{cor:2.9},
$M_0:=\langle X\rangle+d\Delta$ is an $\mathcal{M}_d(\Delta)$-monoid.
By Proposition~\ref{prop:arb-inter}, $M_0$ can be written as a nonempty intersection
of elements of $\mathcal{M}_d(\Delta)$, say
$M_0=\bigcap_{i\in I} S_i$, with $I\neq\varnothing$ and
$S_i\in\mathcal{M}_d(\Delta)$ for all $i\in I$.
Since $X\subseteq M_0\subseteq S_i$ for all $i\in I$, the defining intersection
$N=\mathcal{M}_d(\Delta)[X]$ is contained in each $S_i$, hence
$N\subseteq \bigcap_{i\in I} S_i=M_0$.

Therefore $\mathcal{M}_d(\Delta)[X]=N=\langle X\rangle+d\Delta$.
\end{proof}

\begin{definition}
Let $M$ be an $\mathcal{M}_d(\Delta)$-monoid and let $X\subseteq M$. We say that $X$ is an \emph{$\mathcal{M}_d(\Delta)$-system of generators} of $M$ if $M=\langle X\rangle+d\Delta$.
It is \emph{minimal} if no proper subset of $X$ has this property.
\end{definition}

We will use that every submonoid $M\subseteq(\mathbb{N},+)$ has a unique finite minimal system of generators $\operatorname{msg}(M)$; see \cite[Corollary~2.8]{book}.

\begin{theorem}\label{thm:Md-min-gen}
Let $M$ be an $\mathcal{M}_d(\Delta)$-monoid and set
\[
A=\{x\in \operatorname{msg}(M): x\notin d\Delta\}.
\]
Then $A$ is the unique minimal $\mathcal{M}_d(\Delta)$-system of generators of $M$.
\end{theorem}

\begin{proof}
By Theorem~\ref{th:2.8}, every $\mathcal{M}_d(\Delta)$-monoid $M$ admits a decomposition $M=\langle A\rangle+d\Delta$ with 
$A=\{x\in \operatorname{msg}(M): x\notin d\Delta\}$. In particular, $A$ is an $\mathcal{M}_d(\Delta)$-system of generators of $M$.

To prove minimality and uniqueness, let $B\subseteq M$ be any $\mathcal{M}_d(\Delta)$-system of generators of $M$. By definition,
$M=\langle B\rangle+d\Delta$. Fix $x\in A$. Since $x\in M$, we can write $x=b+u$ with $b\in\langle B\rangle$ and $u\in d\Delta$.
Because $x\notin d\Delta$, necessarily $u\neq x$. If $u>0$, then $x=b+u$ expresses $x$ as a sum of two nonzero elements of $M$
(because $\langle B\rangle\subseteq M$ and $d\Delta\subseteq M$), contradicting $x\in\operatorname{msg}(M)$. Therefore $u=0$ and $x=b\in\langle B\rangle$. Since $x$ is a minimal generator of $M$, this forces $x\in B$. Hence $A\subseteq B$ for every $\mathcal{M}_d(\Delta)$-system of generators $B$ of $M$, so $A$ is the unique minimal one.
\end{proof}

By Theorem~\ref{thm:Md-min-gen}, every \(\mathcal{M}_d(\Delta)\)-monoid \(M\) has a unique
minimal \(\mathcal{M}_d(\Delta)\)-system of generators.

\begin{definition}
Let \(M\) be an \(\mathcal{M}_d(\Delta)\)-monoid. We denote by
$\operatorname{msg}_{\mathcal{M}_d(\Delta)}(M)$ the unique minimal \(\mathcal{M}_d(\Delta)\)-system of generators of \(M\),
and we define the \(\mathcal{M}_d(\Delta)\)-rank of \(M\) as
\[
\operatorname{rank}_{\mathcal{M}_d(\Delta)}(M)
=
\left|\operatorname{msg}_{\mathcal{M}_d(\Delta)}(M)\right|.
\]
\end{definition}

Thus \(\operatorname{rank}_{\mathcal{M}_d(\Delta)}(M)\) measures how many generators outside \(d\Delta\) are needed to recover \(M\) in the fixed-quotient setting. By Theorem~\ref{thm:Md-min-gen}, $\operatorname{msg}_{\mathcal{M}_d(\Delta)}(M) =
\operatorname{msg}(M)\setminus d\Delta$.

\begin{example}\label{ex:Sda-rank}
Let $\Delta$ be a numerical semigroup, let $d\ge 2$, and choose $a\in\Delta\setminus\{0\}$ such that $a+1\in\Delta$.
Set $T:=\Delta_d(a)$, i.e. $T=d\Delta\ \cup\ \bigl(\{da+1,da+2,\dots,da+d-1\}+d\Delta\bigr)$. By Theorem~\ref{th:3.1}, one has $T/d=\Delta$, hence $T\in\mathcal{M}_d(\Delta)$.
Moreover, Proposition~\ref{prop:2} gives
\[
\operatorname{msg}(T)=d\,\operatorname{msg}(\Delta)\ \cup\ \{da+1,\dots,da+d-1\}.
\]
Therefore the unique minimal $\mathcal{M}_d(\Delta)$-system of generators of $T$ is $\{da+1,\dots,da+d-1\}$ (Theorem~\ref{thm:Md-min-gen}), and
$\operatorname{rank}_{\mathcal{M}_d(\Delta)}(T)=d-1$.
In particular, for $d=2$ this construction yields rank-one elements of $\mathcal{M}_2(\Delta)$.
\end{example}

The rank-zero case consists only of the monoid $d\Delta$, which is not a numerical
semigroup when $d\ge2$. Thus the first genuinely numerical layer of the fiber is rank
one. We now characterize this layer completely.

\subsection{Extremal rank and embedding dimension}

We now record two consequences of the description of minimal
\(\mathcal{M}_d(\Delta)\)-systems of generators. We first compute the maximal value of the rank on the fiber, and then relate the relative rank to the embedding dimension.

\begin{proposition}\label{prop:maxrank}
Let \(m=m(\Delta)\). Then
\[
\max\{\operatorname{rank}_{\mathcal{M}_d(\Delta)}(S): S\in
\mathcal{M}_d(\Delta)\}=(d-1)m.
\]
\end{proposition}

\begin{proof}
Let \(S\in\mathcal{M}_d(\Delta)\), and set
\(A=\operatorname{msg}_{\mathcal{M}_d(\Delta)}(S)\). By Theorem~\ref{thm:Md-min-gen},
\(A=\operatorname{msg}(S)\setminus d\Delta\). Since \(dm\in S\), two distinct elements
of \(A\) cannot be congruent modulo \(dm\), for then the larger one would be the sum of
the smaller one and a nonzero element of \(S\). Moreover, no element of \(A\) is divisible
by \(d\), because the elements of \(S\) divisible by \(d\) are precisely those of \(d\Delta\).
Hence \(A\) meets at most the congruence classes modulo \(dm\) which are not divisible
by \(d\), and there are \((d-1)m\) such classes. Thus
\(\operatorname{rank}_{\mathcal{M}_d(\Delta)}(S)=|A|\le (d-1)m\).

Conversely, choose a multiple \(B\) of \(dm\) such that \(B>d\operatorname{F}(\Delta)\), and set
\(X_B=\{B+r:0\le r<dm,\ d\nmid r\}\). Then \(X_B\subseteq\Delta\), since every
element of \(X_B\) is greater than \(\operatorname{F}(\Delta)\). Also \(|X_B|=(d-1)m\) and
\(X_B\cap d\Delta=\emptyset\). Moreover,
\(\langle X_B\rangle\cap d(\mathbb N\setminus\Delta)=\emptyset\): indeed, if
\(z\in\langle X_B\rangle\cap d\mathbb N\) and \(z\ne0\), then \(z\ge B>d\operatorname{F}(\Delta)\);
writing \(z=dn\), we get \(n>\operatorname{F}(\Delta)\), and hence \(n\in\Delta\). Finally,
\(\gcd(X_B\cup d\Delta)=1\), since \(B+1\in X_B\) and \(\gcd(B+1,d)=1\).
Hence, by Theorem~\ref{th:Md-fiber}, \(S_B:=\langle X_B\rangle+d\Delta\) belongs to
\(\mathcal{M}_d(\Delta)\).

We claim that \(X_B\subseteq\operatorname{msg}(S_B)\). Let \(y\in X_B\) and write
\(y=u+d\delta\), with \(u\in\langle X_B\rangle\) and \(\delta\in\Delta\). Since
\(y<B+dm\le 2B\), either \(u=0\) or \(u=x\in X_B\). The case \(u=0\) is impossible
because \(d\nmid y\). If \(u=x\in X_B\), then \(y=x+d\delta\). For \(x\ne y\), this
gives \(0<y-x<dm\), while \(d\delta\) is either \(0\) or at least \(dm\), a contradiction.
Hence \(x=y\) and \(\delta=0\). Thus \(y\) is a minimal generator of \(S_B\).

Therefore \(X_B\subseteq\operatorname{msg}(S_B)\setminus d\Delta\), and by
Theorem~\ref{thm:Md-min-gen},
\(\operatorname{rank}_{\mathcal{M}_d(\Delta)}(S_B)\ge |X_B|=(d-1)m\). Together with
the upper bound, this proves the result.
\end{proof}

The previous proposition shows that the rank filtration has a finite top
layer. The explicit family \(\Delta_d(a)\) introduced in
Section~\ref{Sect3} does not reach this top layer if
\(\Delta\ne\mathbb N\). Indeed, by Example~\ref{ex:Sda-rank},
\(\operatorname{rank}_{\mathcal{M}_d(\Delta)}(\Delta_d(a))=d-1\), whereas
Proposition~\ref{prop:maxrank} shows that the maximal rank is
\((d-1)m(\Delta)\). Since \(\Delta\ne\mathbb N\), we have
\(m(\Delta)>1\), and hence \(d-1 < (d-1)m(\Delta)\).

We next explain how the relative rank controls the ordinary embedding
dimension.

\begin{proposition}\label{prop:4.9}
Let \(S\in\mathcal{M}_d(\Delta)\). If
\(A=\operatorname{msg}_{\mathcal{M}_d(\Delta)}(S)\), then
\begin{enumerate}
    \item \(\operatorname{msg}(S) = A\cup \{dn:n\in\operatorname{msg}(\Delta),\ dn\notin\langle A\rangle\}\).
    \item \(\operatorname{e}(S) = \operatorname{e}(\Delta)+ |A| - \left| \{n\in\operatorname{msg}(\Delta):dn\in\langle A\rangle\} \right|\).
\end{enumerate}
\end{proposition}

\begin{proof}
Write \(\operatorname{msg}(\Delta)=\{n_1,\ldots,n_e\}\). Since \(S=\langle A\rangle+d\Delta\), we have \(S=\langle A,dn_1,\ldots,dn_e\rangle\). By Theorem~\ref{thm:Md-min-gen}, the elements of \(A\) are exactly the minimal generators of \(S\) that do not belong to \(d\Delta\). Thus all possible redundancies in the above generating set occur among the elements \(dn_1,\ldots,dn_e\).

Fix \(i\in\{1,\ldots,e\}\). If \(dn_i\in\langle A\rangle\), then \(dn_i\) is clearly redundant in the generating set \(A\cup d\operatorname{msg}(\Delta)\). Conversely, suppose that \(dn_i\) is redundant. Then \(dn_i=u+d\delta\), for some \(u\in\langle A\rangle\) and some \(\delta\in\langle n_1,\ldots,\widehat{n_i},\ldots,n_e\rangle\). Since the right-hand side is equal to \(dn_i\), the element \(u\) is divisible by \(d\). Moreover, \(u\in S\), and the elements of \(S\) divisible by \(d\) are precisely the elements of \(d\Delta\), because \(S/d=\Delta\). Hence \(u=d\eta\) for some \(\eta\in\Delta\). Therefore \(n_i=\eta+\delta\).

Since \(n_i\) is a minimal generator of \(\Delta\), the equality \(n_i=\eta+\delta\), with \(\eta,\delta\in\Delta\), forces one of \(\eta,\delta\) to be zero. If \(\eta=0\), then \(n_i=\delta\in\langle n_1,\ldots,\widehat{n_i},\ldots,n_e\rangle\), contradicting the minimality of \(n_i\). Hence \(\delta=0\), and then \(\eta=n_i\). Therefore \(dn_i=u\in\langle A\rangle\). Thus \(dn_i\) is redundant if and only if \(dn_i\in\langle A\rangle\).

It follows that the minimal system of generators of \(S\) is exactly \(A\cup \{dn:n\in\operatorname{msg}(\Delta),\ dn\notin\langle A\rangle\}\). Taking cardinalities gives the desired formula for \(\operatorname{e}(S)\).
\end{proof}

The inequalities
\[
\operatorname{rank}_{\mathcal{M}_d(\Delta)}(S)\le \operatorname{e}(S)\le
\operatorname{rank}_{\mathcal{M}_d(\Delta)}(S)+\operatorname{e}(\Delta)
\]
follow immediately from Proposition~\ref{prop:4.9}. The following example shows that both bounds are sharp.

\begin{example}
Let \(\Delta=\langle 3,4,5\rangle\) and \(d=2\). If
\(S=\langle 7,9\rangle+2\Delta\), then
\(S=\langle 6,7,8,9,10\rangle\). Since the even elements of \(S\) are precisely
\(2\Delta\), we obtain \(S/2=\Delta\). Moreover, we have
\(\operatorname{rank}_{\mathcal M_2(\Delta)}(S)=2\), and
\(\operatorname{e}(S)=5=\operatorname{rank}_{\mathcal M_2(\Delta)}(S)+\operatorname{e}(\Delta)\). Thus the upper
bound is attained.

On the other hand, if \(T=\langle 3,5\rangle+2\Delta\), then \(T=\langle 3,5\rangle\),
since \(6=2\cdot 3\), \(8=3+5\), and \(10=2\cdot 5\). Since \(T/2=\Delta\), we have
\(\operatorname{rank}_{\mathcal M_2(\Delta)}(T)=2\), and
\(\operatorname{e}(T)=2=\operatorname{rank}_{\mathcal M_2(\Delta)}(T)\). Thus the lower bound is
attained.
\end{example}

Nevertheless, these inequalities are only a rough consequence of Proposition~\ref{prop:4.9}. For instance, in rank one the formula yields only the two possibilities \(\operatorname{e}(S)=\operatorname{e}(\Delta)\) or \(\operatorname{e}(S)=\operatorname{e}(\Delta)+1\), as will be seen below.

\subsection{Rank-one elements}

We now focus on the first nontrivial case, namely $\mathcal{M}_d(\Delta)$-rank one. For an $\mathcal{M}_d(\Delta)$-monoid $M\neq d\Delta$, we set $\mu(M)=\min(M\setminus d\Delta)$.

\begin{proposition}\label{prop:mu-rank}
Let $M$ be an $\mathcal{M}_d(\Delta)$-monoid and let $A=\operatorname{msg}_{\mathcal{M}_d(\Delta)}(M)$ be its unique minimal $\mathcal{M}_d(\Delta)$-system of generators (Theorem~\ref{thm:Md-min-gen}).
If $M\neq d\Delta$, set $\mu(M)=\min(M\setminus d\Delta)$. Then:
\begin{enumerate}
\item if $M\neq d\Delta$, then $\mu(M)=\min(A)$ (in particular, $\mu(M)\in A$);
\item $\operatorname{rank}_{\mathcal{M}_d(\Delta)}(M)=0$ if and only if $M=d\Delta$;
\item $\operatorname{rank}_{\mathcal{M}_d(\Delta)}(M)=1$ if and only if $M\neq d\Delta$ and $M=\mathcal{M}_d(\Delta)[\{\mu(M)\}]$.
\end{enumerate}
\end{proposition}

\begin{proof}
By definition, $\operatorname{rank}_{\mathcal{M}_d(\Delta)}(M)=|A|$, and by Theorem~\ref{thm:Md-min-gen} we have $M=\langle A\rangle+d\Delta$.

(1) Assume $M\neq d\Delta$ and let $\mu=\mu(M)$. Since $A\subseteq M\setminus d\Delta$, we have
$\min(A)\ge \mu$. If $\mu\notin A$, then $\mu\in M=\langle A\rangle+d\Delta$, so we can write
$\mu=b+u$ with $b\in\langle A\rangle$ and $u\in d\Delta$. If $u>0$, then $b<\mu$ and
$b\in M$. Moreover, $b\notin d\Delta$, for otherwise $\mu=b+u\in d\Delta$, a contradiction.
Hence $b\in M\setminus d\Delta$ with $b<\mu$, contradicting the definition of $\mu$.
Therefore $u=0$, and so $\mu=b\in\langle A\rangle$. Since $\mu>0$, some element of $A$
appears in any nonzero expression of $\mu$ in $\langle A\rangle$, and therefore
$\min(A)\le\mu$. Thus $\mu(M)=\min(A)$.

(2) $\operatorname{rank}_{\mathcal{M}_d(\Delta)}(M)=0$ iff $A=\varnothing$, and then $M=\langle A\rangle+d\Delta=d\Delta$. The converse is immediate.

(3) If $\operatorname{rank}_{\mathcal{M}_d(\Delta)}(M)=1$, then $A=\{\mu(M)\}$ by (1), hence
$M=\langle\mu(M)\rangle+d\Delta=\mathcal{M}_d(\Delta)[\{\mu(M)\}]$ by Proposition~\ref{prop:Md-closure}.
Conversely, if $M\neq d\Delta$ and $M=\mathcal{M}_d(\Delta)[\{\mu(M)\}]$, then $M=\langle\mu(M)\rangle+d\Delta$ (Proposition~\ref{prop:Md-closure}),
so $M$ admits a $\mathcal{M}_d(\Delta)$-system of generators with one element and therefore $\operatorname{rank}_{\mathcal{M}_d(\Delta)}(M)\le 1$.
By (2) it cannot be $0$, hence it is $1$.
\end{proof}

\begin{theorem}\label{th:rankone-char}
The following are equivalent for a numerical semigroup $S$:
\begin{enumerate}
\item $S\in\mathcal{M}_d(\Delta)$ and $\operatorname{rank}_{\mathcal{M}_d(\Delta)}(S)=1$;
\item there exists $x\in\Delta\setminus d\Delta$ with $\gcd(x,d)=1$ such that $S=\langle x\rangle+d\Delta$.
\end{enumerate}
Moreover, in this case $x$ is uniquely determined by $S$ and equals $\mu(S)=\min(S\setminus d\Delta)$.
\end{theorem}

\begin{proof}
Assume (1). By Proposition~\ref{prop:mu-rank}(3), if $x=\mu(S)=\min(S\setminus d\Delta)$ then
$S=\mathcal{M}_d(\Delta)[\{x\}]=\langle x\rangle+d\Delta$. In particular, $x\notin d\Delta$ and $x\in S/d=\Delta$.
If $\gcd(x,d)=g>1$, then $\langle x\rangle+d\Delta\subseteq g\mathbb{N}$, so $\mathbb{N}\setminus S$ is infinite, contradicting that $S$ is a numerical semigroup.
Hence $\gcd(x,d)=1$, proving (2). Uniqueness follows from $x=\mu(S)$.

Conversely, assume (2) and set $S=\langle x\rangle+d\Delta$. Since $\gcd(\{x\}\cup d\Delta)=\gcd(x,d)=1$, Lemma~\ref{lem:2.10} implies that $S$ is a numerical semigroup.
Moreover $S/d=\Delta$: clearly $\Delta\subseteq S/d$. If $y\in S/d$, write $dy=ax+d\delta$ with $a\in\mathbb{N}$ and $\delta\in\Delta$.
Reducing modulo $d$ gives $d\mid ax$, and since $\gcd(x,d)=1$ we get $d\mid a$, say $a=db$. Then $y=bx+\delta\in\Delta$.
Thus $S\in\mathcal{M}_d(\Delta)$.

Finally, since $x\in S\setminus d\Delta$, and every element of $S\setminus d\Delta$
has the form $kx+d\delta$ with $k\ge 1$, every element of $S\setminus d\Delta$ is
at least $x$. Hence $\mu(S)=x$, and
$\operatorname{rank}_{\mathcal{M}_d(\Delta)}(S)=1$ by Proposition~\ref{prop:mu-rank}(3).
\end{proof}

\subsection{Frobenius-type and related invariants in the rank-one case}\label{subsec:rankone-invariants}

From this point on, we work inside $\mathcal{M}_d(\Delta)$, and we compute Frobenius-type invariants for rank-one elements explicitly.

Recall that $S$ is a \emph{gluing} of numerical semigroups $S_1=\langle A\rangle$ and $S_2=\langle B\rangle$ if
$S=\langle d_1A\cup d_2B\rangle$ for some $d_1,d_2\ge2$ with $\gcd(d_1,d_2)=1$, $d_1\in S_2\setminus\operatorname{msg}(S_2)$ and $d_2\in S_1\setminus\operatorname{msg}(S_1)$; see \cite[Chapter~8, Section~1]{book}.

\begin{corollary}\label{cor:gluing-N-Delta}
Let $x\in \Delta$ with $\gcd(x,d)=1$. If, in addition, $x\notin \operatorname{msg}(\Delta)$, then $S=\langle x\rangle+d\Delta$ is a gluing of $\mathbb{N}$ and $\Delta$.
\end{corollary}

\begin{proof}
Let $S_1=\mathbb{N}=\langle 1\rangle$ and $S_2=\Delta$, and take $d_1=x$ and $d_2=d$.
Since $\Delta\ne\mathbb N$ and $x\in\Delta\setminus\operatorname{msg}(\Delta)$ with
$\gcd(x,d)=1$, we have $x\ge2$. Hence $d_1,d_2\ge2$. Moreover,
$\gcd(d_1,d_2)=1$, $d_2=d\in \mathbb{N}\setminus\operatorname{msg}(\mathbb{N})
=\mathbb{N}\setminus\{1\}$, and $d_1=x\in\Delta\setminus\operatorname{msg}(\Delta)$.
Finally, since $d\Delta=\langle d\,\operatorname{msg}(\Delta)\rangle$, we have
\[
S=\langle x\rangle+d\Delta
=\langle \{x\}\cup d\,\operatorname{msg}(\Delta)\rangle
=\langle d_1\operatorname{msg}(\mathbb{N})\cup d_2\operatorname{msg}(\Delta)\rangle,
\]
so $S$ is a gluing of $\mathbb{N}$ and $\Delta$.
\end{proof}

\begin{remark}\label{rem:gluing-quotient}
Corollary~\ref{cor:gluing-N-Delta} is a special case of the following observation.
Let $\Delta$ and $T$ be numerical semigroups, and let $S$ be a gluing of $\Delta$ and $T$ with coefficients $d,e$
(in particular, $\gcd(d,e)=1$ and $e\in\Delta$). Then $S/d=\Delta$.

Indeed, $\Delta\subseteq S/d$ since $d\Delta\subseteq S$. Conversely, if $y\in S/d$, then $dy\in S$ and, grouping the generators
in the gluing description, we may write $dy=d\delta+et$ for some $\delta\in\Delta$ and $t\in T$. Hence $d(y-\delta)=et$, and $\gcd(d,e)=1$ yields $d\mid t$, say $t=du$ with $u\in\mathbb{N}$. Therefore $y=\delta+eu\in\Delta$ because $e\in\Delta$ and $\Delta$ is a submonoid.
\end{remark}

The next result gives closed formulas for $\operatorname{F}(S)$ and $\operatorname{g}(S)$ for $S=\langle x\rangle+d\Delta$, obtained by analyzing residue classes modulo $d$.

\begin{theorem}\label{th:Fg-rankone}
Let $x\in \Delta$ with $\gcd(x,d)=1$, and set $S=\langle x\rangle+d\Delta$. Then
\[
\operatorname{F}(S)=d\operatorname{F}(\Delta)+(d-1)x,
\qquad
\operatorname{g}(S)=d\operatorname{g}(\Delta)+\frac{(d-1)(x-1)}{2}.
\]
\end{theorem}

\begin{proof}
Since $\gcd(x,d)=1$, multiplication by $x$ permutes the residue classes modulo $d$. For each $r\in\{0,1,\dots,d-1\}$ let
$a_r\in\{0,1,\dots,d-1\}$ be the unique integer with $a_r x\equiv r\pmod d$.
We claim that the elements of $S$ in residue class $r$ are exactly $a_r x+d\Delta$.
Indeed, if $\delta\in\Delta$ then $a_r x+d\delta\in S$ and is congruent to $r$ modulo $d$.
Conversely, if $n\in S$ and $n\equiv r\pmod d$, write $n=ax+d\delta$ with $a\in\mathbb{N}$ and $\delta\in\Delta$.
Then $a\equiv a_r\pmod d$, so $a=a_r+dt$ with $t\in\mathbb{N}$ and hence
$n=a_r x+d(tx+\delta)\in a_r x+d\Delta$.

Since the maximal gap of $\Delta$ is $\operatorname{F}(\Delta)$, the largest gap of $S$ in the residue class $r$ modulo $d$ is $a_r x+d\operatorname{F}(\Delta)$.
Therefore
\[
\operatorname{F}(S)=\max_{0\le r\le d-1}\bigl(a_r x+d\operatorname{F}(\Delta)\bigr)=d\operatorname{F}(\Delta)+(d-1)x,
\]
because $\{a_r:0\le r\le d-1\}=\{0,1,\dots,d-1\}$.

For the genus, write $a_r x=r+d q_r$. Since $a_r x\equiv r\pmod d$ and $a_r x\ge0$, we have $q_r\in\mathbb{N}$. As shown above, the elements of $S$ in residue class $r$ are exactly $a_r x+d\Delta$, so an integer $r+dk$ lies in $S$ if and only if $k-q_r\in\Delta$. Hence this residue class contributes exactly $q_r+\operatorname{g}(\Delta)$ gaps, and summing gives
\[
\operatorname{g}(S)=\sum_{r=0}^{d-1}(q_r+\operatorname{g}(\Delta))=d\operatorname{g}(\Delta)+\sum_{r=0}^{d-1}q_r.
\]
Finally, since $a_r x=r+dq_r$,
\[
\sum_{r=0}^{d-1}q_r=\frac{1}{d}\Bigl(x\sum_{r=0}^{d-1}a_r-\sum_{r=0}^{d-1}r\Bigr)
=\frac{1}{d}\Bigl(x\cdot\frac{d(d-1)}{2}-\frac{d(d-1)}{2}\Bigr)=\frac{(d-1)(x-1)}{2},
\]
and the formula for $\operatorname{g}(S)$ follows.
\end{proof}

Recall that $\operatorname{PF}(S)=\{x\in \mathbb{Z}\setminus S:\ x+s\in S \text{ for all } s\in S\setminus\{0\}\}$ and $\operatorname{t}(S)=|\operatorname{PF}(S)|$.
We say that $S$ is \emph{symmetric} if $x\notin S$ implies $\operatorname{F}(S)-x\in S$ for all $x\in\mathbb{Z}$, equivalently $\operatorname{PF}(S)=\{\operatorname{F}(S)\}$; this is equivalent to the associated semigroup ring being Gorenstein \cite{kunz}.

\begin{theorem}\label{th:PF-rankone}
Let $x\in \Delta$ with $\gcd(x,d)=1$, and set $S=\langle x\rangle+d\Delta$. Then $\operatorname{t}(S)=\operatorname{t}(\Delta)$ and
\[
\operatorname{PF}(S)=\{df+(d-1)x : f\in \operatorname{PF}(\Delta)\}.
\]
In particular, $S$ is symmetric if and only if $\Delta$ is symmetric.
\end{theorem}

\begin{proof}
Set $S=\langle x\rangle+d\Delta=\{ax+d\delta:\ a\in\mathbb{N},\ \delta\in\Delta\}$.
For every \(k\in\mathbb{Z}\), we first record the equivalence
\[
(d-1)x+dk\in S \iff k\in\Delta.
\]
If $k\in\Delta$, then $(d-1)x+dk\in\langle x\rangle+d\Delta=S$. Conversely, assume
\((d-1)x+dk\in S\). Then necessarily \((d-1)x+dk\in\mathbb{N}\), and $(d-1)x+dk=ax+d\delta$ for some $a\in\mathbb{N}$ and
$\delta\in\Delta$. Reducing modulo $d$ yields $ax\equiv -x\pmod d$, hence
$a\equiv d-1\pmod d$ since $\gcd(x,d)=1$. Writing $a=d-1+dt$ with
$t\in\mathbb{N}$ and dividing by $d$ gives $k=tx+\delta\in\Delta$.

Fix $f\in\operatorname{PF}(\Delta)$ and set $w=df+(d-1)x$. Since $f\notin\Delta$,
the equivalence implies $w\notin S$. Moreover, since $x\in\Delta\setminus\{0\}$,
the defining property of pseudo-Frobenius numbers gives $f+x\in\Delta$, and therefore
$w+x=d(f+x)\in d\Delta\subseteq S$.

Let $s\in S\setminus\{0\}$. If $s\in d\Delta$, say $s=d\delta$ with
$\delta\in\Delta\setminus\{0\}$, then $f+\delta\in\Delta$ and
$w+s=(d-1)x+d(f+\delta)\in S$. If $s\notin d\Delta$, write
$s=qx+d\delta$ with $q\ge 1$ and $\delta\in\Delta$. Then
$s-x=(q-1)x+d\delta\in S$, so $w+s=(w+x)+(s-x)\in S$. Hence
$w\in\operatorname{PF}(S)$, proving
\[
\{df+(d-1)x:\ f\in\operatorname{PF}(\Delta)\}\subseteq \operatorname{PF}(S).
\]

Conversely, let $w\in\operatorname{PF}(S)$. Since $x\in S\setminus\{0\}$, we have
$w+x\in S$, so $w+x=ax+d\delta$ for some $a\in\mathbb{N}$ and $\delta\in\Delta$.
If $a\ge 1$, then $w=(a-1)x+d\delta\in S$, a contradiction. Hence $a=0$ and
$w+x=d\delta$ for some $\delta\in\Delta$. Write $\delta=x+f$ with
$f\in\mathbb{Z}$, so $w=(d-1)x+df$.

If $f\in\Delta$, then $w\in S$ by the equivalence, again a contradiction; thus
$f\notin\Delta$. Finally, for every $\delta'\in\Delta\setminus\{0\}$ we have
$w+d\delta'\in S$, and
$w+d\delta'=(d-1)x+d(f+\delta')$. Using the equivalence, this forces
$f+\delta'\in\Delta$ for all $\delta'\in\Delta\setminus\{0\}$, hence
$f\in\operatorname{PF}(\Delta)$.

Therefore
\[
\operatorname{PF}(S)=\{df+(d-1)x:\ f\in \operatorname{PF}(\Delta)\}.
\]
The map $f\mapsto df+(d-1)x$ is a bijection, so $\operatorname{t}(S)=\operatorname{t}(\Delta)$.

Finally, a numerical semigroup is symmetric if and only if its type equals $1$,
equivalently $\operatorname{PF}(S)=\{\operatorname{F}(S)\}$. Hence $S$ is symmetric if and only if $\Delta$ is symmetric. 
\end{proof}

\begin{example}
In the setting of Example~\ref{ex:Sda-rank}, take $d=2$. Then $
\Delta_2(a)=2\Delta\ \cup\ \bigl((2a+1)+2\Delta\bigr)=\langle 2a+1\rangle+2\Delta$, so $\Delta_2(a)\in\mathcal{M}_2(\Delta)$ and it has $\mathcal{M}_2(\Delta)$-rank one (Example~\ref{ex:Sda-rank}).
Moreover, the formulas from Section~\ref{Sect3} and the rank-one formulas from Section~\ref{Sect4} agree:
by Theorem~\ref{th:3.3} (with $d=2$), $
\operatorname{F}(\Delta_2(a))=2\operatorname{F}(\Delta)+2a+1$ and $\operatorname{g}(\Delta_2(a))=2\operatorname{g}(\Delta)+a$,
which coincide with Theorem~\ref{th:Fg-rankone} upon setting $x=2a+1$. Likewise, Theorem~\ref{th:PF-rankone} yields
$\operatorname{PF}(\Delta_2(a))=\{\,2f+(2a+1): f\in \operatorname{PF}(\Delta)\,\}$ and $\operatorname{t}(\Delta_2(a))=\operatorname{t}(\Delta)$.
\end{example}

\subsection*{Data availability}\mbox{}\par
\noindent Data sharing is not applicable to this article, as no datasets were generated or analyzed during the present work.

\subsection*{Declarations}\mbox{}\par 
\noindent The first author was partially supported by project PID2022-138906NB-C21 (MCIN /AEI/10.13039/501100011033, NextGenerationEU/PRTR) and grant GR24068 (Junta de Extre\-madura, ERDF). The second author received no specific funding for this work.

\medskip
\noindent The authors have no relevant financial or non-financial interests to disclose.

\medskip
\noindent
All authors have contributed equally in the development of this work.


\end{document}